\documentclass[12pt,a4paper,reqno]{amsart} 
\usepackage[left=2.4cm,right=2.4cm,top=2.8cm,bottom=2.8cm,a4paper]{geometry}
\usepackage{amsmath,amscd,amssymb,latexsym}
\usepackage{graphicx}
\usepackage{enumerate}
\usepackage{stackengine}
\stackMath   
\usepackage{placeins}

\usepackage[all,cmtip]{xy}

\makeatletter \@namedef{subjclassname@2010}{
  \textup{2010} Mathematics Subject Classification}
\makeatother

\usepackage{color}
\definecolor{gray}{rgb}{0.65,0.65,0.65}
\definecolor{ultramarine}{rgb}{0.07, 0.04, 0.56}
\usepackage[pagebackref=true,colorlinks=true,linkcolor=ultramarine,citecolor=ultramarine,urlcolor=cyan]{hyperref}

\begin{document}
\setlength{\unitlength}{1mm}

\newtheorem{thm}{Theorem}[section]
\newtheorem{lem}[thm]{Lemma}
\newtheorem{prop}[thm]{Proposition}
\newtheorem{defn}[thm]{Definition}
\newtheorem{cor}[thm]{Corollary}
\newtheorem{rmk}[thm]{Remark}
\newtheorem{conj}[thm]{Conjecture}

\newcommand{\QQ}{\mathbb{Q}}
\newcommand{\ZZ}{\mathbb{Z}}
\newcommand{\RR}{\mathbb{R}}
\newcommand{\NN}{\mathbb{N}}
\newcommand{\CC}{\mathbb{C}}
\newcommand{\XX}{\mathbb{X}}
\newcommand{\YY}{\mathbb{Y}}
\newcommand{\FF}{\mathbb{F}}

\newcommand{\cO}{\mathcal{O}}
\newcommand{\cC}{\mathcal{C}}
\newcommand{\cN}{\mathcal{N}}

\newcommand{\fa}{\mathfrak{a}}
\newcommand{\fb}{\mathfrak{b}}
\newcommand{\fp}{\mathfrak{p}}
\newcommand{\fF}{\mathfrak{F}}
\newcommand{\fM}{\mathfrak{M}}
\newcommand{\fX}{\mathfrak{X}}
\newcommand{\fY}{\mathfrak{Y}}

\newcommand{\SL}{\mathrm{SL}}
\newcommand{\GL}{\mathrm{GL}}
\newcommand{\sign}{\mathrm{sign}}
\newcommand{\op}{\mathrm{op}}
\newcommand{\cl}{\mathrm{Cl}}
\newcommand{\am}{\mathrm{Am}}

\title[Congruences on the class numbers of $\QQ(\sqrt{\pm 2p})$ for $p\equiv3$ $(\text{mod }4)$ a prime]
{Congruences on the class numbers of $\QQ(\sqrt{\pm 2p})$\\ for $p\equiv3$ $(\text{mod }4)$ a prime}

\author{Jigu Kim and Yoshinori Mizuno}

\keywords{Class numbers, Quadratic fields, Hirzebruch sums}
\subjclass[2010]{Primary 11R29, Secondary 11A55, 11F20}

\maketitle

\noindent{\small {\bf Abstract.}
For a prime $p\equiv 3$ $(\text{mod }4)$, let $h(-8p)$ and $h(8p)$ be the class numbers of $\mathbb{Q}(\sqrt{-2p})$ and $\mathbb{Q}(\sqrt{2p})$, respectively. Let $\Psi(\xi)$ be the Hirzebruch sum of a quadratic irrational $\xi$. We show that $h(-8p)\equiv h(8p)\Big(\Psi(2\sqrt{2p})/3-\Psi(\frac{1+\sqrt{2p}}{2})/3\Big)$ $(\text{mod }16)$. Also, we show that 
$h(-8p)\equiv 2h(8p)\Psi(2\sqrt{2p})/3$ $(\text{mod }8)$ if $p\equiv 3$ $(\text{mod }8)$, and $h(-8p)\equiv \big(2h(8p)\Psi(2\sqrt{2p})/3\big)+4$ $(\text{mod }8)$ if $p\equiv 7$ $(\text{mod }8)$.}

\bigskip

\section{Introduction and results}\label{sec-1}
This paper is in conjunction with \cite{KM1}, and we briefly recall notations as follows. By $\Delta$ we denote a {\it quadratic discriminant}, i.e., a non-square integer which is congruent to $0$ or $1$ $(\text{mod }4)$. We write $\Delta=df^2$, where $d$ is the {\it fundamental discriminant}, i.e., the discriminant of the field $K=\mathbb{Q}(\sqrt{\Delta})$, and $f\in\NN$ is the {\it conductor}. Let $\omega_\Delta:=(\sigma_\Delta+\sqrt{\Delta})/2$, where $\sigma_\Delta=0$ if $\Delta$ is even and $\sigma_\Delta=1$ otherwise. The ring $\cO_{\Delta}=\ZZ[\omega_\Delta]$ (resp., $\cO_{d}=\ZZ[\omega_d]$) is called the {\it quadratic order} with conductor $f$ (resp., the {\it maximal order}) in $K$. We denote the {\it wide} (resp., {\it narrow}) {\it class number} of $\mathcal{O}_{\Delta}$ by $h(\Delta)$ (resp., $h^+(\Delta)$). When $\Delta<0$, $\omega(\Delta)$ denotes the number of the roots of unity in $\mathcal{O}_{\Delta}$, and in the case that $\Delta>0$, $\varepsilon_{\Delta}$ denotes the {\it fundamental unit} of $\mathcal{O}_{\Delta}$. For a quadratic irrational $\xi$ that has a continued fraction expansion $\xi=[\hat{v}_1,\cdots,\hat{v}_k;\overline{v_0,\cdots,v_{l-1}}]$, the {\it Hirzebruch sum} of $\xi$ is defined by
$$\Psi(\xi):=\left\{
\begin{array}{ll}\displaystyle\sum_{i=0}^{l-1}(-1)^{k+i}v_i &\text{if }l\text{ is even},\\
\displaystyle 0 &\text{if }l\text{ is odd}.
\end{array}\right.
$$

Zagier \cite{Zag75} obtained a beautiful formula between the class number of the maximal order in an imaginary quadratic field and the Hirzebruch sums of quadratic irrationals that are representatives of the class group of the maximal order in a real quadratic field; recently, the second named author and Kaneko \cite{KM20} generalized this formula to quadratic orders  (for detail, see Theorem \ref{thm-KM}). As an application, we summarize some congruences modulo 16 for quadratic class numbers: 
\begin{thm}\label{integrated-thm}
Assume that a pair $(d_1,d_2)$ of distinct negative fundamental discriminants is one of the following: 
\begin{enumerate}[(i)]
\item\label{sm-1} 
$(-4,-p)$ for some prime $p\equiv3$ $(\text{mod }4)$,
\item\label{sm-2} 
$(-4,-4p)$ for some prime $p\equiv1$ $(\text{mod }4)$,
\item\label{sm-3}  
$(-p_1,-p_2)$ for some distinct primes $p_1\equiv p_2\equiv 3$ $(\text{mod }4)$.
\end{enumerate}
Then we have 
\begin{equation}\label{integrated-congruences}
24\frac{h(d_1)h(d_2)}{\omega(d_1)\omega(d_2)} \equiv h(d_1d_2)\Psi(\omega_{d_1d_2})
\quad (\text{mod }16).
\end{equation}
\end{thm}

Part \eqref{sm-1} is proved by Chua et al. \cite[Theorem 1.3]{CGPY15} using Zagier's formula \cite{Zag75}. Part \eqref{sm-2} is shown by the second named author \cite[Theorem 1]{Miz21}
due to the generalized formula \cite{KM20} (cf. Remark \ref{remark-formula}). For parts \eqref{sm-1} and \eqref{sm-2},
we see that $h(d_1)=1$ and  ${24}/{(\omega(d_1)\omega(d_2))}=3$ if $p>3$.
For part \eqref{sm-3}, we showed the congruence \eqref{integrated-congruences}  and provided a similar result for the real quadratic order with conductor $2$ in $\QQ(\sqrt{p_1p_2})$ \cite[Theorems 1.3 and 1.4]{KM1}; 
in this case, the congruence \eqref{integrated-congruences} reduces to a congruence modulo $8$ as ${24}/{(\omega(-d_1)\omega(-d_2))}$ is $6$ if $p_1$, $p_2>3$ and 2 otherwise.

Now we state the main results of this paper as follows. 
\begin{thm}\label{main-thm-new}
Let $p$ be a prime such that $p\equiv 3$ $(\text{mod }4)$. Then we have 
$$h(-8p)\equiv  
h(8p)\Big(\frac{\Psi(2\sqrt{2p})}{3}-\frac{\Psi(\frac{1+\sqrt{2p}}{2})}{3}\Big) \quad (\text{mod }16).$$
\end{thm}

\begin{thm}\label{main-thm-new-new}
If $p$ is a prime such that $p\equiv 3$ $(\text{mod }8)$, then 
$$h(-8p)\equiv  2h(8p)\frac{\Psi(2\sqrt{2p})}{3} \quad (\text{mod }8).$$
If $p$ is a prime such that $p\equiv 7$ $(\text{mod }8)$, then 
$$h(-8p)\equiv  2h(8p)\frac{\Psi(2\sqrt{2p})}{3}+4 \quad (\text{mod }8).$$
\end{thm}

\begin{rmk}
{\rm \phantom{1}\\ 
(1) Compared to Theorem \ref{integrated-thm}, Theorems \ref{main-thm-new} and \ref{main-thm-new-new} are in the case where $(d_1,d_2)=(-4,-8p)$ for a prime $p\equiv 3$ $(\text{mod }4)$. In this case, we have $h(32p)\equiv 2$ $(\text{mod }4)$ since $h(8p)$ is odd and $h(32p)=2h(8p)$ for a prime $p\equiv 3$ $(\text{mod }4)$ (cf. \cite[Theorem 5.6.13 and p. 328]{Hal13}). On the other hands, we have that 
$h(-8p)\equiv 2$ $(\text{mod }4)$ if $p\equiv 3$ $(\text{mod }8)$ and $h(-8p)\equiv 0$ $(\text{mod }4)$ if $p\equiv 7$ $(\text{mod }8)$; this is because the $2$-rank of the class group of $\QQ(\sqrt{-2p})$ is $1$ by genus theory, and 
the $4$-rank of that is equal to $1-\text{rank}_{\FF_2}R$, where $\FF_2$ is the finite field of order 2 and 
the {\it R\'{e}dei matrix} $R\in \text{Mat}_{2\times2}(\FF_2)$ is given through the Kronecker symbol $(\tfrac{\phantom{a}}{\phantom{a}})$ and 
$$R={\footnotesize 
\begin{pmatrix}\big(1-(\tfrac{2}{p})\big)/2&& \big(1-(\tfrac{2}{p})\big)/2 \\
\big(1-(\tfrac{-p}{2})\big)/2 &&\big(1-(\tfrac{-p}{2})\big)/2 \end{pmatrix}}
=\left\{\begin{array}{ll}
{\footnotesize 
\begin{pmatrix}1&1\\
1 &1 \end{pmatrix}} &\text{if }\, p\equiv3\,\,(\text{mod }8),\\
{\footnotesize 
\begin{pmatrix}0&0 \\
0 &0 \end{pmatrix}} &\text{if }\, p\equiv3\,\,(\text{mod }8) 
\end{array}\right.
$$
(cf. \cite{Red33}).\\ 
(2) For a prime $p>3$, the Hirzebruch sums $\Psi(2\sqrt{2p})$ 
and $\Psi(\frac{1+\sqrt{2p}}{2})$ are divisible by $3$ (cf. \cite[Satz 2]{Lan76}). For $p=3$, we see that 
$h(-24)=2$, $h(24)=1$, $\Psi(2\sqrt{6})=7$ and $\Psi((1+\sqrt{6})/2)=1$.\\
(3) There have been researches done on the class numbers of $\QQ(\sqrt{\pm 2p})$ modulo powers of 2 in the case of a prime $p\equiv1$ $(\text{mod }4)$; for instance, refer to \cite{KW82}, \cite{Wil81}, and \cite{WF85}.}
\end{rmk}

The paper is organized as follows.
In Section \ref{sec-2}, we recall the Kaneko-Mizuno-Zagier formula, Hirzebruch sums and Dedekind sums. 
In Section \ref{sec-3}, we review some facts on 
the ambiguous classes and the fundamental unit of $\mathcal{O}_{32p}$. 
We prove Theorems \ref{main-thm-new} and \ref{main-thm-new-new} in Section \ref{sec-4}.
\smallskip

\section{Preliminaries}\label{sec-2}
In this section, we recall the Kaneko-Mizuno-Zagier formula and review a calculation method of Hirzebruch sums through Dedekind sums; for more detail, we refer to \cite[Sections 2 and 3.1]{KM1}. 

Let $\Delta$ be a quadratic discriminant and 
$\mathbb{X}_\Delta:=\{\xi=\frac{b+\sqrt{\Delta}}{2a};\, a\neq 0,b,c,\in\mathbb{Z},\,\gcd(a,b,c)=1,\,b^2-4ac=\Delta\}$ the set of all {\it quadratic irrationals} of discriminant $\Delta$. 
For $A={\footnotesize \begin{pmatrix}\alpha& \beta \\
\gamma &\delta \end{pmatrix}}\in\mathrm{GL}_2(\mathbb{Z})$, let $A\xi:=(\alpha \xi+\beta)/(\gamma \xi + \delta)$.
Two quadratic irrationals $\xi$, $\eta\in\mathbb{X}_\Delta$ are {\it equivalent} 
(resp., {\it properly equivalent}) if there exists $M\in\mathrm{GL}_2(\mathbb{Z})$ (resp., $M\in\mathrm{SL}_2(\mathbb{Z})$) such that $\xi=M\eta$; by $\xi\sim \eta$ (resp., $\xi\sim_+ \eta$) we denote this relation.
Let $\mathfrak{X}_\Delta$ (resp., $\mathfrak{X}^+_\Delta$) the set of all equivalence (resp., proper equivalence) classes of $\mathbb{X}_\Delta$. For $\xi=(b+\sqrt{\Delta})/a\in\mathbb{X}_\Delta$, let $\xi':=(-b+\sqrt{\Delta})/(-a)$.
When $\Delta>0$, we say that $\xi\in\mathbb{X}_\Delta$ is {\it reduced} if $\xi>1$ and $-1<\xi'<0$; we denote by $\mathbb{X}^0_\Delta$ the set of all reduced quadratic irrationals of discriminant $\Delta$. For $\xi=\frac{b+\sqrt{\Delta}}{2a}\in\mathbb{X}^0_{\Delta}$, we have that $a$ and $b$ are positive as $\xi>\xi'$ and $\xi+\xi'>0$. We recall properties of quadratic irrationals and their continued fractions as follows: 
\begin{prop}\label{prop-cfrac}
Let $\xi=[u_0,u_1,u_2,\cdots]$ be a continued fraction of an irrational $\xi\in\mathbb{R}\setminus \mathbb{Q}$, and $\xi_n:=[u_n,u_{n+1},u_{n+2},\cdots]$ for $n\in\NN$. Let $\Delta$ be a positive quadratic discriminant. 
\begin{enumerate}[(i)]
\item\label{2.1.1} The sequence $(u_n)_{n\ge 0}$ is ultimately periodic if and only if $\xi$ is a quadratic irrational, 
and it is periodic if and only if $\xi$ is a reduced quadratic irrational. 
\item\label{minus-proper-euqiv}
For $\xi\in\mathbb{X}_\Delta$ and $n\in\NN$, $\xi_n$ is properly equivalent to $(-1)^n\xi$. 
\item\label{all-red-equiv}
Let $\xi=[\hat{v_1},\cdots,\hat{v_k};\overline{v_0,\cdots,v_{l-1}}]\in\mathbb{X}_{\Delta}$. 
The period length $l$ is odd if and only if $\mathcal{N}(\varepsilon_{\Delta})=-1$. The set of all reduced quadratic irrationals which are equivalent to $\xi$ is given by 
$\{\xi_k, \xi_{k+1}, \cdots, \xi_{k+l-1}\}$. 
\item\label{inverse-cfrac}  
If $\xi=[\overline{v_0,v_1,\cdots,v_{l-1}}]$, then $-\xi'^{-1}=[\overline{v_{l-1},v_{l-2},\cdots,v_0}]$.
\end{enumerate}
\end{prop}

Let $\mathbb{X}^+_\Delta:=\{\xi=\frac{b+\sqrt{\Delta}}{2a}\in\XX_\Delta;\, a>0\}$ and $\mathbb{X}^-_\Delta:=\{\xi=\frac{b+\sqrt{\Delta}}{2a}\in\XX_\Delta;\, a<0\}$. For $\Delta>0$, we have $\fX^+_\Delta=\XX^+_\Delta/\sim_+$ since every quadratic irrational is properly equivalent to a reduced one by Proposition \ref{prop-cfrac}. For $\Delta<0$, $\fX^+_\Delta$ is a disjoint union of $\XX^+_\Delta/\sim_+$ and $\XX^-_\Delta/\sim_+$ because for $\xi\in\mathbb{X}^+_\Delta$ and $\eta\in\mathbb{X}^-_\Delta$, we have $\xi\not\sim_+\eta$ by a simple calculation. 

We recall a composition on $\mathbb{X}_\Delta^+$ and an operation on $\mathbb{X}_\Delta^+/\sim$ (or $\XX_{\Delta}^+/\sim_+$) as follows. For $\xi$, $\eta\in\mathbb{X}_{\Delta}^+$, we write $\xi=\frac{b_1+\sqrt{\Delta}}{2a_1}$ and $\eta=\frac{b_2+\sqrt{\Delta}}{2a_2}$. Let $e=\gcd(a_1,a_2,(b_1+b_2)/2)$ and $a_3:=(a_1a_2)/e^2$. Then there is a unique integer $b_3$ modulo $2a_3$ such that 
$$
\left\{
\begin{array}{rl}
b_3\equiv b_1 &(\text{mod }2\frac{a_1}{e}),\\
b_3\equiv b_2 &(\text{mod }2\frac{a_2}{e}),\\
\frac{b_1+b_2}{2e}b_3\equiv \frac{\Delta+b_1b_2}{2e} &(\text{mod }2a_3).
\end{array}
\right.
$$
We define a composition $\ast$ on $\mathbb{X}_\Delta^+$ by $\xi\ast\eta:=\frac{b_3+\sqrt{\Delta}}{2a_3}\in\mathbb{X}_\Delta^+$ (cf. \cite[Theorem 4.10]{Bue89}). We define an operation on $\XX_{\Delta}^+/\sim$ (resp., $\XX_{\Delta}^+/\sim_+$) by 
$[\xi]_\sim [\eta]_\sim:=[\xi\ast\eta]_\sim$ (resp., $[\xi]_{\sim_+} [\eta]_{\sim_+}:=[\xi\ast\eta]_{\sim_+}$);  
then $\XX_{\Delta}^+/\sim$ (resp., $\XX_{\Delta}^+/\sim_+$) is an abelian group, and $[\omega_\Delta]_\sim$ (resp., $[\omega_\Delta]_{\sim_+}$) is its identity element. For $\xi\in\mathbb{X}_\Delta^+$, we see that $-\xi'\in\mathbb{X}_\Delta^+$ and 
$\xi\ast (-\xi')=\omega_\Delta$. Hence, $[\xi]_\sim$ and $[-\xi']_\sim$ (resp., $[\xi]_{\sim_+}$ and $[-\xi']_{\sim_+}$) are inverses of each other in $\XX_{\Delta}^+/\sim$ (resp., $\XX_{\Delta}^+/\sim_+$). It is well-known that $\XX_{\Delta}^+/\sim$ (resp., $\XX_{\Delta}^+/\sim_+$) is isomorphic to the wide (resp., narrow) class group of $\cO_\Delta$ (cf. \cite[Theorems 6.4.2 and 6.4.5]{Hal13}). 

Now we introduce the genus character on $\XX^+_\Delta/\sim_+$ and the main results of \cite{KM20} as follows.
Let $\Delta=d_1d_2f^2$, where $f$ is a positive integer, $d_1$ and $d_2$ are distinct fundamental discriminants, and $d_1=q^\ast_{1}q^\ast_{2}\cdots q^\ast_{m}$ with prime fundamental
discriminants $q^\ast_{i}$. For $\xi=\frac{b+\sqrt{\Delta}}{2a}\in \mathbb{X}_\Delta$, we define 
$$\chi_{d_1,d_2}^{(\Delta)}(\xi):=
\prod_{i=1}^{m}
\chi^{(q^\ast_{i})}(\xi),$$
where 
$$\chi^{(q^\ast_{i})}(\xi):=\left\{
\begin{array}{ll}
\chi_{q^\ast_{i}}(a) &\text{ if }\, \gcd(a,q^\ast_i)=1,\\
\chi_{q^\ast_{i}}(c) &\text{ if }\, \gcd(c,q^\ast_i)=1,
\end{array}\right.$$
with the Kronecker symbol $\chi_{q^\ast_{i}}:=\big(\frac{q^\ast_{i}}{\cdot}\big)$ and $c=(b^2-\Delta)/(4a)\in\ZZ$. 
We define the genus character $\chi_{d_1,d_2}^{(\Delta)}:\XX^+_\Delta/\sim_+ \,\to \{\pm 1\}$ by 
$$\chi_{d_1,d_2}^{(\Delta)}([\xi]_{\sim_+}):=\chi_{d_1,d_2}^{(\Delta)}(\xi),$$ 
and it is a well-defined homomorphism. Also, we can understand that $\Psi([\xi]_{\sim_+}):=\Psi(\xi)$ due to Proposition \ref{prop-cfrac} \eqref{minus-proper-euqiv}. 
\begin{thm}[{\cite{KM20} and \cite[Section 2]{KM1}}] \label{thm-KM}
Let $\Delta=d_1d_2f^2$, where $d_1$ and $d_2$ are distinct negative fundamental discriminants and $f\in\NN$. 
Let $m_p$ be a nonnegative integer such that $p^{m_p}$ is the highest power of a prime $p$ dividing $f$, and let 
$$\displaystyle
\theta(d_1,d_2,f):=\prod_{\substack{p\mid f\\p:\text{ prime}}}
\frac{(1-\chi_{d_1}(p))(1-\chi_{d_2}(p)) - p^{m_p-1}(p-\chi_{d_1}(p))(p-\chi_{d_2}(p))}{
1-p}.$$ 
Then we have $\mathcal{N}(\varepsilon_\Delta)=1$ and 
\begin{eqnarray}\label{KM-equality-2}
24\frac{h(d_1)h(d_2)}{\omega(d_1)\omega(d_2)} \theta(d_1,d_2,f)
=\sum_{[\eta]_\sim\in\mathfrak{X}_\Delta}\chi_{d_1,d_2}^{(\Delta)}([\eta]_{\sim_+})\Psi([\eta]_{\sim_+}).
\end{eqnarray}
\end{thm}
\begin{rmk}\label{remark-formula}
{\rm \phantom{1}\\ 
(1) For $\eta\in\mathbb{X}_\Delta$ such that $\mathcal{N}(\varepsilon_\Delta)=1$, $[\eta]_\sim$ is a disjoint union of 
$[\eta]_{\sim_+}$ and $[-\eta]_{\sim_+}$, and 
$$\chi_{d_1,d_2}^{(\Delta)}([-\eta]_{\sim_+})\Psi([-\eta]_{\sim_+})
=\big(-\chi_{d_1,d_2}^{(\Delta)}([\eta]_{\sim_+})\big)\big(-\Psi([\eta]_{\sim_+})\big)
=\chi_{d_1,d_2}^{(\Delta)}([\eta]_{\sim_+})\Psi([\eta]_{\sim_+}).
$$ 
(2) Zagier showed the equation \eqref{KM-equality-2} when $d_1$ and $d_2$ are relatively prime negative fundamental discriminants and $\Delta$ is a positive fundamental discriminant \cite[(23)]{Zag75}.}
\end{rmk}

We recall a relationship between Hirzebruch sums and Dedekind sums as follows. For $y\in\RR$, let 
$$((y)):=\left\{\begin{array}{ll}
y-\lfloor y \rfloor -\tfrac{1}{2} &\text{if }\, y\not\in\mathbb{Z},\\
0 &\text{if }\, y\in\mathbb{Z},
\end{array}\right.
$$
where $\lfloor y \rfloor$ is the greatest integer not greater than $y$. For $h\in\ZZ$ and $k\in\NN$ such that $\gcd(h,k)=1$, we define the {\it Dedekind sum} $s(h,k)$ by 
$$s(h,k):=\sum_{m=1}^{k}\Big(\Big(\frac{hm}{k}\Big)\Big) \Big(\Big(\frac{m}{k}\Big)\Big).$$
Then it is well-known that 
\begin{equation}\label{3.2.5-new} 
6ks(h,k)\in\ZZ,
\end{equation}
\begin{equation}\label{3.2.4-new}
s(h,k)=s(h',k) \,\,\text{ if }\, hh'\equiv 1\,\,(\text{mod }k),
\end{equation}
\begin{equation}\label{3.2.1-new} 
s(-h,k)=-s(h,k),
\end{equation}
\begin{equation}\label{3.2.3-new}
s(h,k)+s(k,h)=\frac{h^2+k^2+1}{12hk}-\frac{1}{4} 
\,\,\text{ if }\, h\in\NN,
\end{equation}
and for the {\it Jacobi symbol} $\Big(\frac{\phantom{a}}{\phantom{a}}\Big)$, 
\begin{equation}\label{3.2.6-new} 
\Big(\frac{h}{k}\Big)=(-1)^{\frac{1}{2}\left(\frac{k-1}{2}-6ks(h,k)\right)}\,\,\text{ if }\, k \,\text{ is odd}.
\end{equation}
For a matrix $M={\footnotesize \begin{pmatrix}x& y \\ z &w\end{pmatrix}}\in\SL_2(\ZZ)$ such that $z\neq0$, we define 
$$n_M:=\frac{x+w}{z}-\sign(z)(3+12s(w,|z|)),$$
where $\sign(z):=z/|z|$.
\begin{lem}\label{matrix-lem}
Let $M={\footnotesize \begin{pmatrix}x&y\\z&w\end{pmatrix}}\in\SL_2(\ZZ)$ such that the trace of $M>2$. Then $n_M=n_{M'}$ if one of the following holds:
\begin{enumerate}[(i)]
\item\label{matrix-lem-1} The matrix $M'$ is a $\SL_2(\ZZ)$-conjugate of $M$, 
\item\label{matrix-lem-2} The matrix $M'={\footnotesize \begin{pmatrix}3&-1\\1&0\end{pmatrix}} M$ and $x>0$.
\end{enumerate}
\end{lem}
\begin{proof}
\eqref{matrix-lem-1} Note that $yz\neq 0$ and the trace is invariant under conjugation. 
Since $\SL_2(\ZZ)$ is generated by $T:={\footnotesize \begin{pmatrix}1&1\\0&1\end{pmatrix}}$ and $S:={\footnotesize \begin{pmatrix}0&-1\\1&0\end{pmatrix}}$, it suffices to show that 
$n_M=n_{TMT^{-1}}$ and $n_M=n_{SMS^{-1}}$. We have $TMT^{-1}={\footnotesize
\begin{pmatrix}x+z&-x+y-z+w\\z&-z+w\end{pmatrix}}$; thus $n_{TMT^{-1}}=\frac{x+w}{z}-\sign(z)(3+12s(-z+w,|z|))=n_M$. 
We have $SMS^{-1}={\footnotesize \begin{pmatrix}w&-z \\ -y&x\end{pmatrix}}$. If $x=0$, then $y=-z\in\{\pm1\}$; hence, we get $n_M=n_{SMS^{-1}}$. Now we assume that $x\neq0$. By \eqref{3.2.4-new}, \eqref{3.2.1-new} and \eqref{3.2.3-new}, we have that
$s(w,|z|)=s(x,|z|)=\sign(x)s(|x|,|z|)=\sign(x)\big(-s(|z|,|x|)+\frac{x^2+z^2+1}{12|xz|}-\frac{1}{4}\big)$ and 
$n_{M}=\frac{x+w}{z}-\frac{x^2+z^2+1}{xz}+12\sign(x)s(z,|x|)-3\big(1-\sign(x)\big)\sign(z)$. 
Similarly, we have that $s(x,|y|)=\sign(x)s(|x|,|y|)=\sign(x)\big(-s(|y|,|x|)+\frac{x^2+y^2+1}{|xy|}-\frac{1}{4}\big)$ and 
$n_{SMS^{-1}}=-\frac{x+w}{y}+\frac{x^2+y^2+1}{xy}+3\big(1-\sign(x)\big)\sign(y)$. 
If $x<0$, then $w>0$ and $yz=xw-1<0$; so, we have $-\big(1-\sign(x)\big)\sign(z)=\big(1-\sign(x)\big)\sign(y)$ for $x\neq0$. 
Therefore, we get $n_{M}-n_{SMS^{-1}}=\frac{(y+z)(xw-yz-1)}{xyz}=0$.\\ 
\eqref{matrix-lem-2} We refer to \cite[Lemma 8]{KM20}.
\end{proof}
\smallskip

For $\Delta>0$ and $\xi\in\mathbb{X}_\Delta$, we define $n(\xi):=n_{M_\xi}$, where $M_\xi$ is the unique integral matrix
$M_\xi:={\footnotesize \begin{pmatrix}x& y \\ z &w\end{pmatrix}}$ such that $\varepsilon_\Delta \xi=x\xi+y$ 
and $\varepsilon_\Delta=z\xi+w$. For $\xi=\frac{b+\sqrt{\Delta}}{2a}\in\mathbb{X}_\Delta$, $M_\xi$ is given by 
\begin{equation}\label{omega-2}
M_\xi=
{\footnotesize \begin{pmatrix} 1 & (b-\sigma_\Delta)/2 \\ 0 & a\end{pmatrix}}
{\footnotesize \begin{pmatrix} q+r\sigma_\Delta & r(\Delta-\sigma_\Delta)/4 \\ r & q\end{pmatrix}}
{\footnotesize \begin{pmatrix} 1 & (b-\sigma_\Delta)/2 \\ 0 & a\end{pmatrix}}^{-1}, 
\end{equation}
where $q$ and $r$ are positive integers such that $\varepsilon_\Delta=q+r\omega_\Delta$. 
\begin{lem}\label{HSDS-new}
Let $\Delta$ be a positive quadratic discriminant such that $\cN(\varepsilon_\Delta)=1$. Then for any $\eta\in\mathbb{X}_\Delta$, $\Psi(\eta)=n(\eta)$.
\end{lem}
\begin{proof} 
By Proposition \ref{prop-cfrac} there exist $A\in\SL_2(\ZZ)$ and $\xi\in\mathbb{X}_\Delta^0$ such that $\eta=A\xi$.  
One can easily check that $M_{\eta}=AM_{\xi}A^{-1}$. We have $\Psi(\eta)=\Psi(\xi)=n(\xi)$ by \cite[Lemma 3.6]{KM1}.  
Also, we get $n(\eta)=n_{M_\eta}=n_{AM_{\xi}A^{-1}}=n_{M_{\xi}}=n(\xi)$ by Lemma \ref{matrix-lem} \eqref{matrix-lem-1}. Therefore, $\Psi(\eta)=n(\eta)$. 
\end{proof}
\smallskip

\section{The quadratic order $\cO_{32p}$ for a prime $p\equiv3$ $(\text{mod }4)$}\label{sec-3}

We first recall some general facts on class groups of quadratic orders as follows. Let $\Delta=df^2$ be a quadratic discriminant and $\cl(\Delta)$ the wide class group of $\cO_\Delta$. We define a group homomorphism $\phi_\Delta:\cl(d)\to \cl(\Delta)$ by $[\fa]\to[\fa\cap\cO_{\Delta}]$ if $\fa+f\cO_d=\cO_d$. We denote by $\hat{\phi}_\Delta:\cl(\Delta)\to\cl(d)$ a group homomorphism sending $[\hat{\fa}]$ to $[\hat{\fa} \cO_{d}]$. It is well-known that $\hat{\phi}_\Delta \circ \phi_\Delta$ is the identity on $\cl(d)$, and $\hat{\phi}_\Delta$ is surjective (cf. \cite[Assertions A and B in p. 183 and Theorem 5.9.7]{Hal13}). We define the {\it ambiguous class group} of $\cO_\Delta$ by $\am(\Delta):=\cl(\Delta)[2]$. An ideal $(a,\frac{b+\sqrt{\Delta}}{2})$ in $\cO_{\Delta}$ is called an {\it ambiguous ideal} if $(a,\frac{b+\sqrt{\Delta}}{2})=(a,\frac{-b+\sqrt{\Delta}}{2})$. We can always take $b$ such that $0\le b <2a$; in this case, an ambiguous ideal $(a,\frac{b+\sqrt{\Delta}}{2})$ satisfies either $b=0$ or $b=a$. Also, it is known that if $\Delta>0$, $\cN(\varepsilon_\Delta)=1$ and $\Delta$ is not of the form $4x^2+y^2$ with relatively prime positive integers $x$, $y$, then every ambiguous class contains an ambiguous ideal (cf. \cite[Theorem 5.6.11]{Hal13}). 

Let $p$ be a prime such that $p\equiv3$ $(\text{mod }4)$. Then we have that $h(8p)$ is odd and $\cN(\varepsilon_{8p})=1$. 
We write $\varepsilon_{8p}=t+u\sqrt{2p}$ and $\varepsilon_{32p}=q+r\sqrt{8p}$, where $t$, $u$, $q$ and $r\in\NN$. 
Note that $t$ is odd and $u$ is even as $t^2-2pu^2=1$; hence $\varepsilon_{8p}=\varepsilon_{32p}\in\cO_{32p}$. 
We have $h(32p)=2h(8p)/(\cO_{8p}^\times:\cO_{32p}^\times)=2h(8p)$ (cf. \cite[p. 328]{Hal13}). For the following split exact sequence 
$$\xymatrix{
1 \ar[r] &\ker{\hat{\phi}_{32p}}
\ar[r] &\cl(32p)\ar[r]^{\hat{\phi}_{32p}} &\cl(8p) \ar[r] &1,}$$
we have $\ker{\hat{\phi}_{32p}}=\am(32p)$ because $\vert\ker{\hat{\phi}_{32p}}\vert=2$ 
and $h(32p)\equiv 2$ $(\text{mod }4)$. Hence we get  
$$\cl(32p)=\am(32p)\oplus \phi_{32p}(\cl(8p)).$$ 

Now we present explicit representatives for ambiguous classes as follows. Since $p\equiv3$ $(\text{mod }4)$, $32p$ is not of the form $4x^2+y^2$ with relatively prime positive integers $x$, $y$; then every ambiguous class contains an ambiguous ideal. 
There are $8$ ambiguous ideals in $\cO_{32p}$ as follows: 
$$(1,\sqrt{8p}),\,(8,\sqrt{8p}),\,(p,\sqrt{8p}),\,(8p,\sqrt{8p}),\,$$
$$(4,2+\sqrt{8p}),\,(8,4+\sqrt{8p}),\,(4p,2p+\sqrt{8p}),\,(8p,4p+\sqrt{8p}).$$
\begin{lem}\label{am-rep}
For a prime $p\equiv3$ $(\text{mod }4)$, the ambiguous class group $\am(32p)$ of $\cO_{32p}$ is given by 
$$\am(32p)=\big\{\big[(1,\sqrt{8p})\big],\,\big[(4,2+\sqrt{8p})\big]\big\}.$$
\end{lem}
\begin{proof}
It suffices to show that $\sqrt{8p}\not\sim (2+\sqrt{8p})/4$ in $\mathbb{X}_{32p}$ 
since $\mathfrak{X}_{32p}$ corresponds to $\cl(32p)$ and 
$\left\vert\am(32p)\right\vert=2$. 
We suppose that there exists a matrix ${\footnotesize \begin{pmatrix}a& b \\ c&d\end{pmatrix}}$
$\in\GL_2(\ZZ)$ such that $(a\sqrt{8p}+b)/(c\sqrt{8p}+d)=(2+\sqrt{8p})/4$. 
Then we have 
\begin{align}
\label{eq-1-new}\left\vert ad-bc \right\vert =1,\\
\label{eq-2-new}4a=2c+d,\\
\label{eq-3-new}4b=8pc+2d.
\end{align}
By \eqref{eq-2-new} we have that $d$ is even, and we write $d=2\delta$. By \eqref{eq-1-new} we get that $bc$ is odd; by \eqref{eq-2-new} we see that $\delta$ is odd. Substituting \eqref{eq-2-new} and \eqref{eq-3-new} into \eqref{eq-1-new}, we have 
$$\left\vert \delta^2-2pc^2\right\vert=1.$$ 
Since $c$ and $\delta$ are odd, we have that $\delta^2-2pc^2\equiv3$ $(\text{mod }8)$; so, it is a contradiction.
\end{proof}
\begin{rmk}
{\rm For an ambiguous ideal $(a,\frac{b+\sqrt{\Delta}}{2})$ such that  $0\le b<2a$, let $(\check{a},\check{b}):=(\frac{|\Delta|}{4a},0)$ if $b=0$, and let $(\check{a},\check{b}):=(\frac{|\Delta|}{a},\frac{|\Delta|}{a})$ if $a=b$. We call $\big((a,\frac{b+\sqrt{\Delta}}{2}),\,(\check{a},\frac{\check{b}+\sqrt{\Delta}}{2})\big)$ an ambiguous pair. We note that $\big[(a,\frac{b+\sqrt{\Delta}}{2})\big]=\big[(\check{a},\frac{\check{b}+\sqrt{\Delta}}{2})\big]$ in $\cl(\Delta)$. If $\Delta>0$, $\cN(\varepsilon_\Delta)=1$ and $\Delta$ is not of the form $4x^2+y^2$ with relatively prime integers $x$, $y$,  then every ambiguous class contains exactly two ambiguous pairs (cf. \cite[Theorem 5.6.11]{Hal13}). We have that $\sqrt{8p}\not\sim\sqrt{8p}/8$ if $p\equiv3$ $(\text{mod }8)$, and $\sqrt{8p}\not\sim(4+\sqrt{8p})/8$ if $p\equiv7$ $(\text{mod }8)$ in a similar way to the above proof. Therefore, we get that in $\am(32p)\subset\cl(32p)$ 
$$
\left\{
\begin{array}{l}
\big[(1,\sqrt{8p})\big]=\big[(8,4+\sqrt{8p})\big]=\big[(4p,2p+\sqrt{8p})\big]=\big[(8p,\sqrt{8p})\big],\\
\big[(4,2+\sqrt{8p})\big]=\big[(8,\sqrt{8p})\big]=\big[(p,\sqrt{8p})\big]=\big[(8p,4p+\sqrt{8p})\big]
\end{array}\right. 
\text{if }p\equiv3 \,(\text{mod }8), \,
\text{and}$$
$$
\left\{
\begin{array}{l}
\big[(1,\sqrt{8p})\big]=\big[(8,\sqrt{8p})\big]=\big[(p,\sqrt{8p})\big]=\big[(8p,\sqrt{8p})\big],\\
\big[(4,2+\sqrt{8p})\big]=\big[(8,4+\sqrt{8p})\big]=\big[(4p,2p+\sqrt{8p})\big]=\big[(8p,4p+\sqrt{8p})\big]
\end{array}\right. 
\text{if }p\equiv7 \,(\text{mod }8).$$}
\end{rmk}
\smallskip

Let $p\equiv3$ $(\text{mod }4)$ be a prime, $\omega_{32p}=\sqrt{32p}/2$ and $\omega_{32p}^\am:=(4+\sqrt{32p})/8$.
\begin{lem}\label{am-irr}
Following the above notation, let $\eta=\frac{b+\sqrt{32p}}{2a}\in\mathbb{X}_{32p}^+$ such that $a$ is odd. 
If $\eta\ast\omega_{32p}^\am=\frac{b_\am+\sqrt{32p}}{2a_\am}$, then $a_\am=4a$ and $b_\am\equiv 4$ $(\text{mod }8)$,  
where the composition $\ast$ on $\mathbb{X}_{32p}^+$ is defined in Section \ref{sec-2}. 
\end{lem}
\begin{proof}
Since $\gcd(a,4,(b+4)/2)=1$, we have that $a_\am=4a$, and $b_\am$ is a unique solution modulo $8a$ such that 
$x\equiv b$ $(\text{mod }2a)$, $x\equiv 4$ $(\text{mod }8)$, and $\frac{b+4}{2}x\equiv \frac{32p+4b}{2}$ $(\text{mod }8a)$.
Therefore, we get that $b_\am\equiv 4$ $(\text{mod }8)$.
\end{proof}

Let $p\equiv3$ $(\text{mod }4)$ be a prime, and let $\varepsilon_{32p}=q+r\sqrt{8p}$ be the fundamental unit of $\cO_{32p}$. We recall that $\varepsilon_{8p}=\varepsilon_{32p}$.
\begin{lem}\label{ZY-thm-new} 
Following the above notation,
$2\varepsilon_{32p}=(X+Y\sqrt{2p})^2$ for some integers $X$ and $Y$.
\end{lem}
\begin{proof} 
By \cite[Lemma 3.2]{ZY14} there exist $X$, $Y\in\mathbb{Q}$ such that $2(q+r\sqrt{8p})=(X+Y\sqrt{2p})^2$. Then we get that $2q=X^2+2pY^2$ and $2r=XY$. We have $2pY^4-2qY^2+4r^2=0$ and $Y^2=(q+\alpha)/(2p)$, where $\alpha\in\{\pm1\}$ as $q^2-8pr^2=1$. Since the prime factorization (allowing negative prime powers) of $Y^2$ has only even exponents, $2p$ divides $q+\alpha$; then we have $Y^2\in\ZZ$ and $X^2=2q-2pY^2\in\ZZ$. Therefore, $X$ and $Y$ are integers.
\end{proof}

\begin{lem}[{\cite[Theorem 2]{Wil15}}]\label{Wil-lemma}
Following the above notation, 
$r\equiv 1$ $(\text{mod }2)$ if $p\equiv 3$ $(\text{mod }8)$ and $r\equiv 0$ $(\text{mod }2)$ if $p\equiv 7$ $(\text{mod }8)$.
\end{lem}
\smallskip

\section{Proofs of Theorems \ref{main-thm-new} and \ref{main-thm-new-new}}\label{sec-4}
In this section, we show lemmas that are parellel to those in \cite{CGPY15}, \cite{KM1} and \cite{Miz21}, and we prove Theorems \ref{main-thm-new} and \ref{main-thm-new-new}. 

Let $p\equiv3$ $(\text{mod }4)$ be a prime, and let $\omega_{32p}=\sqrt{32p}/2$ and $\omega_{32p}^{\am}=(4+\sqrt{32p})/8$.  
We denote by $\varepsilon_{32p}=q+r\sqrt{8p}$ the fundamental unit of $\cO_{32p}$.
\begin{lem}\label{pf-lem-1-new}
Following the above notation, for any equivalence class $[\eta]_\sim$ in $\fX_{32p}$, there exists a representative
$\frac{b+\sqrt{32p}}{2a}\in\mathbb{X}^0_{32p}$ of $[\eta]_\sim$
such that $a$ is an odd integer. 
\end{lem}
\begin{proof} 
Let $\xi=\frac{\beta+\sqrt{32p}}{2\alpha}\in\mathbb{X}^0_{32p}$ such that $\xi\sim \eta$ and $\gamma:=\frac{\beta^2-32p}{4\alpha}\in\mathbb{Z}$. 
With the same notation as in Proposition \ref{prop-cfrac}, $\xi_1=\frac{\beta_1+\sqrt{32p}}{2\alpha_1}\in\mathbb{X}^0_{32p}$ is given by 
$$\xi_1
= \frac{1}{\xi-\lfloor \xi \rfloor}
=\frac{2\alpha}{\beta-2\alpha\lfloor \xi \rfloor+\sqrt{32p}}
=\frac{-\beta+2\alpha\lfloor \xi \rfloor+\sqrt{32p}}{-2\gamma+2\beta\lfloor \xi \rfloor - 2\alpha \lfloor \xi \rfloor^2}.
$$
At least one of $\{\alpha,\alpha_1\}$ is odd because $\alpha_1=-\gamma+\beta\lfloor \xi \rfloor-\alpha\lfloor \xi \rfloor^2$, $\gcd(\alpha,\beta,\gamma)=1$ and $\beta$ is even. 
\end{proof}

For a positive quadratic discriminant $\Delta$ and $\eta\in\mathbb{X}_\Delta$, let $\eta^\op:=\lfloor \eta \rfloor-\eta'$. We have that $\eta^\op\sim_+ -\eta'$; so, $[\eta]_{\sim_+}$ and $[\eta^\op]_{\sim_+}$ are inverses of each other in $\XX_\Delta^+/\sim_+$.
\begin{lem}\label{pf-lem-2-new}
Following the above notation,
let $\eta\in\XX_{32p}$.
Then we have
$$\chi^{(32p)}_{-4,-8p}([\eta]_{\sim_+})\Psi([\eta]_{\sim_+})
=\chi^{(32p)}_{-4,-8p}([\eta^\op]_{\sim_+})\Psi([\eta^\op]_{\sim_+}).$$
\end{lem}
\begin{proof} 
By Lemma \ref{pf-lem-1-new} we can take $\xi=\frac{b+\sqrt{32p}}{2a}\in \mathbb{X}^0_{32p}$ such that $\xi\sim \eta$ and $a$ is odd. The period length of $\xi$ is even as $\cN(\varepsilon_{32p})=1$. Let $\xi=[\overline{v_0,v_1,\cdots,v_{2m-1}}]$. 
Then $\xi^\op=\frac{(2a\lfloor \xi \rfloor-b)+\sqrt{32p}}{2a}$ satisfies that $\xi^\op\sim \eta^\op$ 
and $\xi^\op=[\overline{v_0,v_{2m-1},v_{2m-2},\cdots,v_2,v_1}]$ by Proposition \ref{prop-cfrac} \eqref{inverse-cfrac}. Hence, we get that 
$\chi^{(32p)}_{-4,-8p}([\eta]_{\sim_+})\Psi([\eta]_{\sim_+})
=\chi_{-4}(a)\sum_{i=0}^{2m-1}(-1)^i v_i
=\chi^{(32p)}_{-4,-8p}([\eta^\op]_{\sim_+})\Psi([\eta^\op]_{\sim_+})$.
\end{proof}

\begin{lem}\label{pf-lem-3-new}
Following the above notation, let $\eta\in\XX_{32p}$. 
\begin{enumerate}[(i)]
\item\label{lem-4.4.1} If $p\equiv3$ $(\text{mod }8)$, then 
$$\Psi([\omega_{32p}]_{\sim_+}) \equiv \chi^{(32p)}_{-4,-8p}([\eta]_{\sim_+})\Psi([\eta]_{\sim_+})\quad (\text{mod }8).$$
\item\label{lem-4.4.2} If $p\equiv7$ $(\text{mod }8)$, then 
\begin{eqnarray*}
&&\Psi([\omega_{32p}]_{\sim_+})+
\chi^{(32p)}_{-4,-8p}([\omega^\am_{32p}]_{\sim_+})
\Psi([\omega^\am_{32p}]_{\sim_+})\\
&\equiv& 
\chi^{(32p)}_{-4,-8p}([\eta]_{\sim_+})\Psi([\eta]_{\sim_+})+
\chi^{(32p)}_{-4,-8p}([\eta\ast \omega^\am_{32p}]_{\sim_+})
\Psi([\eta\ast \omega^\am_{32p}]_{\sim_+})
\quad (\text{mod }8)
\end{eqnarray*}
and  
$$\Psi([\omega_{32p}]_{\sim_+}) \equiv 
\chi^{(32p)}_{-4,-8p}([\omega_{32p}^\am]_{\sim_+})\Psi([\omega_{32p}^\am]_{\sim_+})+4 \quad (\text{mod }8).$$
\end{enumerate}
\end{lem}
\begin{proof} 
We recall $\cN(\varepsilon_{32p})=1$ and $\varepsilon_{32p}=q+r\sqrt{8p}$ such that $q$ is odd.
Let $\xi=\frac{b+\sqrt{32p}}{2a}\in \mathbb{X}_{32p}$ such that $\xi\sim \eta$, $q+rb/2>0$, and $a$ is a positive odd integer;  such $\xi$ always exists due to Lemma \ref{pf-lem-1-new}. 
Note that
$$\chi^{(32p)}_{-4,-8p}([\omega^\am_{32p}]_{\sim_+})
=\chi^{(32p)}_{-4,-8p}\big(\tfrac{4+\sqrt{32p}}{8}\big)
=\chi_{-4}(1-2p)=-1$$
and
$$
\chi^{(32p)}_{-4,-8p}([\xi\ast \omega^\am_{32p}]_{\sim_+})
=\chi^{(32p)}_{-4,-8p}([\xi]_{\sim_+})\chi^{(32p)}_{-4,-8p}([\omega^\am_{32p}]_{\sim_+})
=-\chi_{-4}(a).
$$
Let $c=\frac{b^2-32p}{4a}\in\mathbb{Z}$, and let $m$ be an integer such that $b=2m$. Let $\xi\ast\omega_{32p}^\am=\frac{b_\am+\sqrt{32p}}{2a_\am}$, where $a_\am=4a$ and $b_\am=2m_\am$ such that $0\le m_\am<4a$ and $m_\am\equiv 2$ $(\text{mod }4)$ by Lemma \ref{am-irr}. Let $c_\am=\frac{b_\am^2-32p}{4a_\am}\in\ZZ$. By Lemma \ref{ZY-thm-new} there exist integers $T$ and $U$ such that $2\varepsilon_{32p}=2(q+r\sqrt{8p})=(T+U\sqrt{2p})^2$; so, $2q=T^2+2pU^2$ and $2r=TU$. Then we have that
$$2(q-rm)=T^2+2pU^2-TU\frac{b}{2}
=\Big(T-\frac{Ub}{4}\Big)^2-\frac{U^2ac}{4},$$
$$2(q+rm)=T^2+2pU^2+TU\frac{b}{2}
=\Big(T+\frac{Ub}{4}\Big)^2-\frac{U^2ac}{4},$$
and 
$$2(q+rm_\am)=T^2+2pU^2+TU\frac{b_\am}{2}=\Big(T+\frac{Ub_\am}{4}\Big)^2-U^2ac_\am.$$
Since $a$ is odd, we have
\begin{eqnarray}\label{cal-Jacobi-symbol}
\Big(\frac{q-rm}{a}\Big)=\Big(\frac{q+rm}{a}\Big)=\Big(\frac{q+rm_\am}{a}\Big)=\Big(\frac{2}{a}\Big)
=(-1)^{\frac{a^2-1}{8}}.
\end{eqnarray}
Let 
$A:={\footnotesize \begin{pmatrix}1& m \\ 0&1\end{pmatrix}}$, 
$A_\am:={\footnotesize \begin{pmatrix}1& (m_\am-2)/4\\ 0&1\end{pmatrix}}$, 
and  $B:={\footnotesize \begin{pmatrix}3& -1 \\ 1&0\end{pmatrix}}$. 
By \eqref{omega-2} we have
\begin{align*}
AM_{\omega_{32p}}A^{-1}
&=\begin{pmatrix}1&m\\0&1\end{pmatrix}
\begin{pmatrix}q&8pr\\r&q\end{pmatrix}
\begin{pmatrix}1&-m\\0&1\end{pmatrix}
=\begin{pmatrix}q+rm& -rac\\r&q-rm\end{pmatrix},
\allowdisplaybreaks\\
BAM_{\omega_{32p}}A^{-1}
&=\begin{pmatrix}3(q+rm)-r &-3rac-(q-rm)\\
q+rm &-rac \end{pmatrix},\allowdisplaybreaks\\
M_{\xi}
&=\begin{pmatrix}1&m\\0&a\end{pmatrix}
\begin{pmatrix}q&8pr\\r&q\end{pmatrix}
\begin{pmatrix}1&-\frac{m}{a}\\0&\frac{1}{a}\end{pmatrix}
=\begin{pmatrix}q+rm &-rc\\ra&q-rm\end{pmatrix},
\allowdisplaybreaks\\
BM_{\xi}
&=\begin{pmatrix}3(q+rm)-ra &-3rc-(q-rm)\\
q+rm &-rc\end{pmatrix},
\allowdisplaybreaks\\
M_{\omega_{32p}^{\am}}
&=\begin{pmatrix}1&2\\0&4\end{pmatrix}
\begin{pmatrix}q&8pr\\r&q\end{pmatrix}
\begin{pmatrix}1&-\frac{1}{2}\\0&\frac{1}{4}\end{pmatrix}
=\begin{pmatrix}q+2r &r(2p-1)\\4r&q-2r\end{pmatrix},
\allowdisplaybreaks\\
A_\am M_{\omega_{32p}^\am}A_\am^{-1}
&=\begin{pmatrix}1&\frac{m_\am-2}{4}\\0&1\end{pmatrix}
\begin{pmatrix}q+2r&r(2p-1)\\4r&q-2r\end{pmatrix}
\begin{pmatrix}1&-\frac{m_\am-2}{4}\\0&1\end{pmatrix}\allowdisplaybreaks\\
&=\begin{pmatrix}q+rm_\am& -rac_\am\\
4r&q-rm_\am\end{pmatrix},\allowdisplaybreaks\\
BA_\am M_{\omega_{32p}^\am}A_\am^{-1}
&=\begin{pmatrix}3(q+rm_\am)-4r &-3rac_\am-(q-rm_\am)\\
q+rm_\am &-rac_\am \end{pmatrix},
\allowdisplaybreaks\\
M_{\xi\ast \omega_{32p}^\am}
&=\begin{pmatrix}1&m_\am\\0&4a\end{pmatrix}
\begin{pmatrix}q&8pr\\r&q\end{pmatrix}
\begin{pmatrix}1&-\frac{m_\am}{4a}\\0&\frac{1}{4a}\end{pmatrix}
=\begin{pmatrix}q+rm_\am &-rc_\am\\
4ra&q-rm_\am\end{pmatrix},\allowdisplaybreaks\\
BM_{\xi\ast\omega_{32p}^\am}
&=\begin{pmatrix}3(q+rm_\am)-4ra &-3rc_\am-(q-rm_\am)\\
q+rm_\am &-rc_\am\end{pmatrix}.
\end{align*}
Note that $q+rm>0$ and $q+rm_\am>0$. By Lemma \ref{matrix-lem} and Lemma \ref{HSDS-new}, we have
\begin{align}
\label{last-eq-1-new}
\Psi(\omega_{32p})
&=n(\omega_{32p})=n_{AM_{\omega_{32p}}A^{-1}}=n_{BAM_{\omega_{32p}}A^{-1}},\allowdisplaybreaks\\
\label{last-eq-2-new}
\Psi(\xi)
&=n(\xi)=n_{M_{\xi}}=n_{BM_{\xi}},\allowdisplaybreaks\\
\label{last-eq-3-new}
\Psi(\omega_{32p}^\am)
&=n(\omega_{32p}^\am)
=n_{BA_\am M_{\omega_{32p}^\am}A_\am^{-1}},\allowdisplaybreaks\\
\label{last-eq-4-new}
\Psi(\xi\ast\omega_{32p}^\am)
&=n(\xi\ast\omega_{32p}^\am)=n_{BM_{\xi\ast\omega_{32p}^\am}}.
\end{align}
\noindent
{\bf{Case 1.}} Suppose that $p\equiv 3$ $(\text{mod }8)$; hence, $r$ is odd by Lemma \ref{Wil-lemma}. 
For \eqref{lem-4.4.1}, by \eqref{last-eq-1-new} and \eqref{last-eq-2-new}, 
it suffices to show that
$$\frac{2q}{r}-3-12s(q-rm,r) \equiv \chi_{-4}(a)\Big(\frac{2q}{ra}-3-12s(q-rm,ra)\Big) \quad (\text{mod }8).$$
Since $a\equiv \chi_{-4}(a)$ $(\text{mod }4)$ for odd $a$, we have that 
$$\frac{2q}{r}-3-\chi_{-4}(a)\Big(\frac{2q}{ra}-3\Big) \equiv 3(\chi_{-4}(a)-1) \quad (\text{mod }8).$$
Therefore, by \eqref{3.2.5-new} it suffices to show that 
$$-6s(q-rm,r)+6\chi_{-4}(a)s(q-rm,ra) \equiv \frac{\chi_{-4}(a)-1}{2} \quad (\text{mod }4).$$
By \eqref{3.2.6-new} we have
\begin{equation}\label{suffice-eq-new-1}
\Big(\frac{q-rm}{r}\Big)=(-1)^{\frac{1}{2}\left(\frac{r-1}{2}-6rs(q-rm,r)\right)}
\end{equation}
and
\begin{equation}\label{suffice-eq-new-2}
\Big(\frac{q-rm}{ra}\Big)=(-1)^{\frac{1}{2}\left(\frac{ra-1}{2}-6ras(q-rm,ra)\right)}.
\end{equation}
By dividing \eqref{suffice-eq-new-1} by \eqref{suffice-eq-new-2}, we get that
\begin{equation}\label{suffice-eq-new-3}
\Big(\frac{q-rm}{a}\Big)=(-1)^{\frac{1}{2}rZ}=(-1)^{\frac{1}{2}Z},
\end{equation}
where $Z=(1-a)/2-6s(q-rm,r)+6\chi_{-4}(a)s(q-rm,ra)$ because $r$ is odd and $a\equiv \chi_{-4}(a)$ $(\text{mod }4)$. By \eqref{cal-Jacobi-symbol} and \eqref{suffice-eq-new-3}, we have 
$$
-6s(q-rm,r)+6\chi_{-4}(a)s(q-rm,ra)\equiv\frac{a-1}{2}+\frac{a^2-1}{4}
\equiv \frac{\chi_{-4}(a)-1}{2}
\quad (\text{mod }4)$$
as desired. 
\smallskip

\noindent
{\bf{Case 2.}} Suppose that $p\equiv 7$ $(\text{mod }8)$; so, $r$ is even by Lemma \ref{Wil-lemma}. Note that $q+rm$ and $q+rm_\am$ are odd. For the first congruence of \eqref{lem-4.4.2}, by \eqref{last-eq-1-new}, \eqref{last-eq-2-new}, \eqref{last-eq-3-new} and \eqref{last-eq-4-new}, it suffices to show that
\begin{eqnarray*}
&&\frac{-r-rac}{q+rm}-12s(-rac,q+rm)
-\frac{-4r-rac_\am}{q+rm_\am}+12s(-rac_\am,q+rm_\am)\\
&\equiv&\chi_{-4}(a)\Big(\frac{-ra-rc}{q+rm}-12s(-rc,q+rm)
-\frac{-4ra-rc_\am}{q+rm_\am}+12s(-rc_\am,q+rm_\am)\Big)\\
&& (\text{mod }8).
\end{eqnarray*}
We have that both
\begin{equation}\label{explicit-1}
\frac{-r-rac}{q+rm}
-\chi_{-4}(a)\frac{-ra-rc}{q+rm}
=-r(a-\chi_{-4}(a)) \frac{c-\chi_{-4}(a)}{q+rm}
\end{equation}
and 
\begin{equation*}
-\frac{-4r-rac_\am}{q+rm_\am}
+\chi_{-4}(a)\frac{-4ra-rc_\am}{q+rm_\am}
=r(a-\chi_{-4}(a))\frac{c_\am-4\chi_{-4}(a)}{q+rm_\am}
\end{equation*}
are divisible by $8$ since $r$ is even and $a\equiv \chi_{-4}(a)$ $(\text{mod }4)$. Thus, by \eqref{3.2.5-new} it suffices to show that 
\begin{eqnarray}\label{2-wts-eq} 
&&-6s(-rac,q+rm)+6\chi_{-4}(a)s(-rc,q+rm)\nonumber\\
&\equiv&
-6s(-rac_\am,q+rm_\am)+6\chi_{-4}(a)s(-rc_\am,q+rm_\am)
\quad (\text{mod }4).
\end{eqnarray}
By \eqref{3.2.6-new} we have
\begin{align}
\label{2-suffice-eq-new-1}
\Big(\frac{-rac}{q+rm}\Big)
&=(-1)^{\frac{1}{2}\left(\frac{q+rm-1}{2}-6(q+rm)s(-rac,q+rm)\right)},
\allowdisplaybreaks\\
\label{2-suffice-eq-new-2}
\Big(\frac{-rc}{q+rm}\Big)
&=(-1)^{\frac{1}{2}\left(\frac{q+rm-1}{2}-6(q+rm)s(-rc,q+rm)\right)}, 
\allowdisplaybreaks\\
\label{2-suffice-eq-new-3}
\Big(\frac{-rac_\am}{q+rm_\am}\Big)
&=(-1)^{\frac{1}{2}\left(\frac{q+rm_\am-1}{2}-6(q+rm_\am)s(-rac_\am,q+rm_\am)\right)},
\allowdisplaybreaks\\
\label{2-suffice-eq-new-4}
\Big(\frac{-rc_\am}{q+rm_\am}\Big)
&=(-1)^{\frac{1}{2}\left(\frac{q+rm_\am-1}{2}-6(q+rm_\am)s(-rc_\am,q+rm_\am)\right)}.
\end{align}
By dividing or multiplying \eqref{2-suffice-eq-new-1} by \eqref{2-suffice-eq-new-2}, we have 
\begin{equation*}
\Big(\frac{a}{q+rm}\Big)=(-1)^{\frac{1}{2}Z_1},
\end{equation*}
where $Z_1=(1-\chi_{-4}(a))(q+rm-1)/2-6(q+rm)s(-rac,q+rm)+6\chi_{-4}(a)(q+rm)s(-rc,q+rm)$. By \eqref{cal-Jacobi-symbol} and the law of quadratic reciprocity, 
\begin{equation*}
Z_1\equiv \frac{(a-1)(q+rm-1)}{2}+\frac{a^2-1}{4} \quad (\text{mod }4).
\end{equation*}
Thus, we have
\begin{eqnarray}\label{2-suffice-eq-new-5}
&&-6(q+rm)s(-rac,q+rm)+6\chi_{-4}(a)(q+rm)s(-rc,q+rm)
\nonumber\\
&\equiv& (a+\chi_{-4}(a)-2)\frac{q+rm-1}{2}+\frac{a^2-1}{4} 
\nonumber\\
&\equiv& \frac{a^2-1}{4} \qquad\qquad (\text{mod }4)
\end{eqnarray}
because $a+\chi_{-4}(a)-2\equiv 0$ $(\text{mod }4)$. By dividing or multiplying \eqref{2-suffice-eq-new-3} by \eqref{2-suffice-eq-new-4}, we have 
\begin{equation*}
\Big(\frac{a}{q+rm_\am}\Big)=(-1)^{\frac{1}{2}Z_2},
\end{equation*}
where $Z_2=(1-\chi_{-4}(a))(q+rm_\am-1)/2-6(q+rm_\am)s(-rac_\am,q+rm_\am)+6\chi_{-4}(a)(q+rm_\am)s(-rc_\am,q+rm_\am)$. By \eqref{cal-Jacobi-symbol} and the law of quadratic reciprocity,
\begin{equation*}
Z_2\equiv \frac{(a-1)(q+rm_\am-1)}{2}+\frac{a^2-1}{4} \quad (\text{mod }4).
\end{equation*}
Similarly, we have
\begin{eqnarray}\label{2-suffice-eq-new-6}
&&-6(q+rm_\am)s(-rac_\am,q+rm_\am)+6\chi_{-4}(a)(q+rm_\am)s(-rc_\am,q+rm_\am)
\nonumber\\
&\equiv& \frac{a^2-1}{4} \qquad\qquad (\text{mod }4).
\end{eqnarray}
Therefore, we get \eqref{2-wts-eq} by \eqref{2-suffice-eq-new-5} and \eqref{2-suffice-eq-new-6}. 

Now we show the second congruence of \eqref{lem-4.4.2} as follows. 
We take $\xi=-1/\omega_{32p}^\am=\frac{-4+\sqrt{32p}}{2(2p-1)}$; then $\xi$ is equivalent to $\omega_{32p}^\am$ and satisfies 
$q+rb/2=q-2r>\varepsilon'_{32p}>0$. By \eqref{explicit-1} and \eqref{2-suffice-eq-new-5}, we have that 
$$\Psi([\omega_{32p}]_{\sim_+})
-\chi_{-4,-8p}^{(32p)}([\omega_{32p}^\am]_{\sim_+})
\Psi([\omega_{32p}^\am]_{\sim_+})
\equiv \frac{(2p-1)^2-1}{2}
\equiv 2p(p-1)
\equiv 4 \quad (\text{mod }8)$$
as desired.
\end{proof}
\smallskip

\begin{proof}[Proofs of Theorems \ref{main-thm-new} and \ref{main-thm-new-new}]\, 
By Lemma $\ref{pf-lem-3-new}$ we have that 
$$
\begin{array}{lllll}
\Psi(\omega_{32p})\equiv -\Psi(\omega_{32p}^\am) \quad &(\text{mod }8) 
&&\text{if } p\equiv 3 &(\text{mod }8),\\
\Psi(\omega_{32p})\equiv -\Psi(\omega_{32p}^\am)+4 \quad &(\text{mod }8) 
&&\text{if } p\equiv 7 &(\text{mod }8).
\end{array}
$$
Thus, Theorem \ref{main-thm-new} implies Theorem \ref{main-thm-new-new} since $h(8p)$ is odd.
Now we show Theorem \ref{main-thm-new} as follows. 
We recall that $\mathfrak{X}_{32p}=\{\omega_{32p},\omega_{32p}^{\am}\}\oplus \phi_{32p}(\mathfrak{X}_{8p})$, where we understand that a map $\phi_{32p}:\mathfrak{X}_{8p}\to\mathfrak{X}_{32p}$ corresponds to the map $\phi_{32p}:\cl(8p)\to\cl(32p)$ that is defined in Section \ref{sec-3}.
By applying Theorem \ref{thm-KM} to $(d_1,d_2)=(-4,-8p)$, we have that
\begin{eqnarray}\label{app-KM-new}
&&h(-8p)\nonumber\\
&=&\frac{1}{3}\sum_{[\eta]_\sim\in\mathfrak{X}_{32p}}\chi_{-4,-8p}^{(32p)}([\eta]_{\sim_+})\Psi([\eta]_{\sim_+})\nonumber\\
&=&\sum_{[\eta]_\sim\in\phi_{32p}(\mathfrak{X}_{8p})}\bigg(\chi_{-4,-8p}^{(32p)}([\eta]_{\sim_+})\frac{\Psi([\eta]_{\sim_+})}{3}
+\chi_{-4,-8p}^{(32p)}([\eta\ast\omega_{32p}^\am]_{\sim_+})\frac{\Psi([\eta\ast\omega_{32p}^\am]_{\sim_+})}{3}
\bigg).
\end{eqnarray}
For $[\eta]_\sim\in \mathfrak{X}_{32p}$, we note that $[\eta^\op]_\sim$ is the inverse of $[\eta]_\sim$  in $\fX_{32p}$. 
Also, for $[\xi]_\sim\in \phi_{32p}(\fX_{8p})$ such that $[\xi]_\sim\neq [\omega_{32p}]_\sim$, we have that $[\xi]_\sim\neq[\xi^\op]_\sim$ in $\phi_{32p}(\fX_{8p})$ because $\phi_{32p}(\fX_{8p})$ is a group of odd order. 
Therefore, there exist quadratic irrationals $\omega_{32p}$, $\xi_{(1)}$, $\xi_{(2)}$, $\cdots$, $\xi_{(m)}$ such that 
$$\phi_{32p}(\fX_{8p})=\{[\omega_{32p}]_\sim, [\xi_{(1)}]_\sim,[\xi_{(1)}^\op]_\sim,\cdots,[\xi_{(m)}]_\sim,[\xi_{(m)}^\op]_\sim\},$$ 
where $h(8p)=2m+1$. It follows that
\begin{align*}
\fX_{32p}=\{ &[\omega_{32p}]_\sim, [\xi_{(1)}]_\sim,[\xi_{(1)}^\op]_\sim,\cdots,
[\xi_{(m)}]_\sim,[\xi_{(m)}^\op]_\sim,\\
 &[\omega_{32p}^\am]_\sim, 
 [\xi_{(1)}\ast\omega_{32p}^\am]_\sim,[\xi_{(1)}^\op\ast\omega_{32p}^\am]_\sim,\cdots,[\xi_{(m)}\ast\omega_{32p}^\am]_\sim,[\xi_{(m)}^\op\ast\omega_{32p}^\am]_\sim\}
\end{align*} 
as $\fX_{32p}=\{[\omega_{32p}]_\sim,[\omega_{32p}^\am]_\sim\}\oplus \phi_{32p}(\fX_{8p})$.
By \eqref{app-KM-new} and Lemma \ref{pf-lem-2-new}, we have 
\begin{eqnarray*}
&&h(-8p)-h(8p)\Big(\frac{\Psi(\omega_{32p})}{3}-\frac{\Psi(\omega_{32p}^\am))}{3}\Big)\\
&=&2\Bigg(\sum_{i=1}^m
\bigg(\chi_{-4,-8p}^{(32p)}([\xi_{(i)}]_{\sim_+})\frac{\Psi([\xi_{(i)}]_{\sim_+})}{3}
+\chi_{-4,-8p}^{(32p)}([\xi_{(i)}\ast\omega_{32p}^\am]_{\sim_+})
\frac{\Psi([\xi_{(i)}\ast\omega_{32p}^\am]_{\sim_+})}{3}\\
&&-\frac{\Psi([\omega_{32p}]_{\sim_+})}{3}
-\chi_{-4,-8p}^{(32p)}([\omega_{32p}^\am]_{\sim_+})\frac{\Psi([\omega_{32p}^\am]_{\sim_+})}{3}\bigg)
\Bigg).
\end{eqnarray*}
Therefore, by Lemma $\ref{pf-lem-3-new}$ we get 
$$h(-8p)- h(8p)\Big(\frac{\Psi(\omega_{32p})}{3}-\frac{\Psi(\omega_{32p}^\am)}{3}\Big) 
\equiv 0 \quad (\text{mod }16)$$
as desired. 
\end{proof}

\bigskip

\section{A conjectural congruence}  
In Appendix, we have observed that a stronger congruence than Theorem \ref{main-thm-new} holds for every prime $p$ such that $p\equiv 3$ $(\text{mod }4)$ and $p\le 1000$. We were led to the following conjecture:  
\begin{conj}\label{conj-new}
Let $p$ be a prime such that $p\equiv 3$ $(\text{mod }4)$. Then we have 
$$h(-8p)\equiv  
h(8p)\Big(\frac{\Psi(2\sqrt{2p})}{3}-\frac{\Psi(\frac{1+\sqrt{2p}}{2})}{3}\Big) \quad (\text{mod }32).$$
\end{conj}

\bigskip

\section*{Appendix}

Let $p\equiv 3$ $(\text{mod }4)$ be a prime. For small primes $p$, we list the factorizations of 
$$\begin{array}{l}
\displaystyle H_1(8p)
:=h(8p)\Big(\frac{\Psi(2\sqrt{2p})}{3}-\frac{\Psi(\frac{1+\sqrt{2p}}{2})}{3}\Big)-h(-8p) \,\text{ and }\\
\displaystyle H_2(8p)
:=2h(8p)\frac{\Psi(2\sqrt{2p})}{3}-h(-8p).
\end{array}$$ 
Here, $h(\Delta)$ is the class number of the quadratic order with discriminant $\Delta$, and 
$\Psi$ is the Hirzebruch sum. 

\smallskip

\FloatBarrier
\begin{table}[!h]
{Table A1. $H_1(8p)$ and $H_2(8p)$ for $p\equiv 3$ $(\text{mod }4)$ such that $h(8p)=1$, $p\le 23$}
{\tabcolsep6pt\begin{tabular}{ccccccccll}
\hline
$p$ 
&$p$ $(\text{mod }8)$
&$\Psi(2\sqrt{2p})$ 
&$\Psi(\frac{1+\sqrt{2p}}{2})$ 
&$h(8p)$ 
&$h(-8p)$ 
&&$H_1(8p)$
&&$H_2(8p)$\\ 
\hline
3 &3 &7 &1  &1 &2 
&&0 &&$2^3\cdot3^{-1}$\\
\hline
7 &7 &12 &0  &1 &4 
&&0 &&$2^2$\\
\hline
11 &3 &15 &9  &1 &2 
&&0 &&$2^3$\\
\hline
19 &3 &21 &3  &1 &6 
&&0 &&$2^3$\\
\hline
23 &7 &24 &12  &1 &4 
&&0 &&$2^2$\\
\hline
\end{tabular}}{}
\end{table}

\FloatBarrier
\begin{table}[!h]%
{Table A2. $H_1(8p)$ and $H_2(8p)$ for $p\equiv 3$ $(\text{mod }8)$ such that $h(8p)\neq1$, $p\le1000$}
{\tabcolsep6pt\begin{tabular}{cccccllll}
\hline
$p$ 
&$\Psi(2\sqrt{2p})$ 
&$\Psi(\frac{1+\sqrt{2p}}{2})$ 
&$h(8p)$ 
&$h(-8p)$ 
&&$H_1(8p)$
&&$H_2(8p)$\\ 
\hline
163 &63 &9  &3 &22 
&&$2^5$ &&$2^3\cdot13$\\
\hline
467 &105 &15 &3  &26 
&&$2^6$ &&$2^3\cdot 23$\\
\hline
491 &111 &9 &5  &10 
&&$2^5\cdot5$ &&$2^3\cdot 3^2\cdot 5$\\
\hline
563 &105 &15 &5  &22 
&&$2^7$ &&$2^3\cdot 41$\\
\hline
739 &147 &21 &3  &30 
&&$2^5\cdot 3$ &&$2^3\cdot 3\cdot 11$\\
\hline
827 &141 &27 &9  &22 
&&$2^6\cdot 5$ &&$2^3\cdot 103$\\
\hline
883 &147 &21 &5  &50 
&&$2^5\cdot 5$ &&$2^3\cdot 5\cdot 11$\\
\hline
\end{tabular}}{}
\end{table}

\FloatBarrier
\begin{table}[!h]%
{Table A3. $H_1(8p)$ and $H_2(8p)$ for $p\equiv 7$ $(\text{mod }8)$ such that $h(8p)\neq1$, $p\le1000$}
{\tabcolsep6pt\begin{tabular}{cccccllll}
\hline
$p$ 
&$\Psi(2\sqrt{2p})$ 
&$\Psi(\frac{1+\sqrt{2p}}{2})$ 
&$h(8p)$ 
&$h(-8p)$ 
&&$H_1(8p)$
&&$H_2(8p)$\\ 
\hline
71 &48 &12  &3 &4 
&&$2^5$ &&$2^2\cdot23$\\
\hline
127 &66 &18  &3 &16 
&&$2^5$ &&$2^2\cdot29$\\
\hline
647 &156 &48  &7 &28 
&&$2^5\cdot 7$ &&$2^2\cdot 5^2\cdot 7$\\
\hline
743 &156 &48  &5 &20 
&&$2^5\cdot 5$ &&$2^2\cdot 5^3$\\
\hline
823 &168 &60  &3 &44
&&$2^6$ &&$2^2\cdot 73$\\
\hline
967 &192 &60  &7 &52
&&$2^8$ &&$2^2\cdot 211$\\
\hline
\end{tabular}}{}
\end{table}

\bigskip

\noindent Institute of Mathematical Sciences,

\noindent Ewha Womans University,

\noindent Seoul, Korea

\noindent E-mail: jigu.kim.math@gmail.com\bigskip

\noindent Graduate School of Technology Industrial and Social Sciences,

\noindent Tokushima University,

\noindent Tokushima, Japan

\noindent E-mail: mizuno.yoshinori@tokushima-u.ac.jp

\end{document}